\documentclass[12pt,a4paper,english,reqno]{amsart}
\usepackage[a4paper,footskip=1.5em]{geometry}
\usepackage{amsmath,amssymb,amsthm,mathtools,comment}
\usepackage[mathscr]{euscript}
\usepackage{tikz,babel,enumitem,adjustbox,}
\usepackage[final]{microtype}
\usepackage[numbers]{natbib}
\usetikzlibrary{decorations.pathreplacing,calc,arrows,arrows.meta,patterns,ipe}
\usepackage{hyperref}
\hypersetup{colorlinks=true,linkcolor=blue,citecolor=blue,pdfpagemode=UseNone,pdfstartview={XYZ null null 1.00}}
\usepackage{cmtiup}
\usepackage{amsfonts}
\usepackage{graphicx}
\usepackage{caption}

\allowdisplaybreaks

\theoremstyle{plain}
\newtheorem{theorem}{Theorem}[section]		
\newtheorem{lemma}[theorem]{Lemma}

\newtheorem{proposition}[theorem]{Proposition}

\theoremstyle{remark}

\def\Prob{\mathbb{P}}
\def\E{\mathbb{E}}
\def\XX{\mathcal{X}}
\def\N{\mathbb{N}}

\def\cond{\,|\,}

\let\emptyset\varnothing

\let\originalleft\left
\let\originalright\right
\renewcommand{\left}{\mathopen{}\mathclose\bgroup\originalleft}
\renewcommand{\right}{\aftergroup\egroup\originalright}

\makeatletter
\def\imod#1{\allowbreak\mkern10mu({\operator@font mod}\,\,#1)}
\makeatother

\title{Slowdown for the geodesic-biased random walk}
\author{Mikhail Beliayeu}
\address{Department of Applied Mathematics, Charles University, Prague, Czech Republic}
\email{mikhail.beliayeu@gmail.com}

\author{Petr Chmel}
\address{Department of Applied Mathematics, Charles University, Prague, Czech Republic}
\email{petr@chmel.net}

\author{Bhargav Narayanan}
\address{Department of Mathematics, Rutgers University, Piscataway, NJ 08854, USA}
\email{narayanan@math.rutgers.edu}

\author{Jan Petr}
\address{Department of Applied Mathematics, Charles University, Prague, Czech Republic}
\email{xpetj01@gmail.com}

\date{15 August 2019}
\subjclass[2010]{Primary 60G50; Secondary 60J10, 60C05}

\begin{document}

\maketitle
\begin{abstract}
	Given a connected graph $G$ with some subset of its vertices excited and a fixed target vertex, in the \emph{geodesic-biased random walk} on $G$, a random walker moves as follows: from an unexcited vertex, she moves to a uniformly random neighbour, whereas from an excited vertex, she takes one step along some fixed shortest path towards the target vertex. We show, perhaps counterintuitively, that the geodesic-bias can slow the random walker down exponentially: there exist connected, bounded-degree $n$-vertex graphs with excitations where the expected hitting time of a fixed target is at least $\exp (\sqrt[4]{n} / 100)$.
\end{abstract}
\section{Introduction}
In this paper, we investigate a model of excited random walk on a connected graph, namely \emph{geodesic-biased random walk}, where the excitations are designed to decrease the hitting time of a fixed target vertex. The model originates in the theoretical computer science and computational biology communities~\cite{back1, back2, back3}, and was brought to our attention by Sousi~\citep{perla}. By way of context, let us mention that various matters relating to hitting times --- recurrence and return times~\citep{intro, recur0, recur1, recur2}, speed~\citep{speed1, speed2, speed3} and slowdown~\citep{slow1, slow2} --- have been investigated in a number of different models of excited random walk; for a broad overview, see~\citep{survey1,survey2}.

Geodesic-biased random walk is defined on a connected $n$-vertex graph $G$. Having fixed a starting vertex $a \in V(G)$, a target vertex $b \in V(G)$ and a subset $\XX \subset V(G)$ of excited vertices, a random walker walks from $a$ until she hits $b$ as follows: from an unexcited vertex of $G$, she moves to a uniformly random neighbour, whereas from an excited vertex, she takes one step along some predetermined shortest path to the target vertex $b$. Our focus here is the hitting time $\tau_a (b, \XX)$ i.e., the first time at which the walker hits $b$ starting from $a$ when the set of excited vertices is $\XX$.

When every vertex is excited, i.e., $\XX = V(G)$, the geodesic-biased walk reduces to a deterministic walk along a shortest path to the target vertex, in which case we have $\E[\tau_a(b, V(G))] = O(n)$. On the other hand, when no vertices are excited, i.e., $\XX = \emptyset$, the geodesic-biased walk reduces to the simple random walk on $G$, and an old result of Lawler~\citep{Lawler} gives a uniform polynomial bound (see also~\citep{KarpLovasz, BrightwellWinkler}) for the expected hitting time of $\E[\tau_a(b, \emptyset)] = O(n^3)$. Many of the existing results in the literature~\citep{back1, back2, back3} show that the expected hitting time of a fixed target in the geodesic-biased walk, for various graphs $G$ and random choices of the set $\XX$ of excited vertices, is significantly smaller than Lawler's uniform bound. Motivated by this, we shall investigate how much the geodesic-bias can decrease the hitting time of a fixed target.

While the geodesic-bias ostensibly aims to decrease hitting times, it is actually not hard to construct examples where the expected hitting time of a fixed target in the geodesic-biased walk is \emph{slightly larger} than the expected hitting time in the analogous simple random walk. To wit, consider a graph where two vertices $a$ and $b$ are connected by two paths of lengths $2$ and $3$, with the middle vertex of the shorter path being attached to a `trap', say a large clique; here, it is not hard to see that exciting $a$ increases the expected hitting time of $b$, since the random walker ends up spending more time in the `trap'. However, the digraph formed by taking a shortest path from each vertex to a fixed target is acyclic, so one cannot string together multiple such `traps' in a cyclic fashion; in particular, such constructions cannot hope to slow the geodesic-biased walk down by more than a constant factor in comparison to the simple random walk.

In the light of the above discussion, it is natural to ask if the results in~\citep{back1, back2, back3} are indicative of a broader phenomenon, and if there is a uniform polynomial bound for the expected hitting time of a fixed target in the geodesic-biased walk, much like Lawler's bound~\citep{Lawler} for the simple random walk. Our first result shows, perhaps surprisingly, that this is not the case: even a single excitation can cause an exponential slowdown.

\begin{theorem}
	\label{thm:unbounded}
	For infinitely many $n \in \N$, there exists a connected graph $G$ on $n$ vertices with $a,b\in V(G)$ such that
	\[\E[\tau_a(b, \{a\})]  = \Omega\left(\exp\left(\frac{\sqrt[4]{n} \log n}{100}\right)\right).\]
\end{theorem}

The construction proving Theorem~\ref{thm:unbounded} produces graphs of unbounded degree. In the context of the simple random walk, bounded-degree graphs are known to behave somewhat differently from those of unbounded degree; for example, as shown by Lawler~\citep{Lawler}, expected hitting times in a bounded-degree $n$-vertex graph are $O(n^2)$. Our second result, also in the spirit of Theorem~\ref{thm:unbounded}, shows that exponential slowdown is unavoidable on graphs of bounded degree as well, though more excitations are required in this case.
\begin{theorem}
	\label{thm:bounded}
	For infinitely many $n \in \N$, there exists a connected graph $G$ on $n$ vertices of maximum degree $3$ with $a,b\in V(G)$ and a set $\XX \subset V(G)$ of $O(\sqrt{n})$ excited vertices such that
	\[\E[\tau_a(b, \XX)]  = \Omega\left(\exp\left(\frac{\sqrt[4]{n}}{100}\right)\right).\]
\end{theorem}

This paper is organised as follows. We give the proofs of Theorems~\ref{thm:unbounded} and~\ref{thm:bounded} in Section~\ref{s:proof}. We conclude with a discussion of some open problems in Section~\ref{s:conc}.

\section{Proofs of the main results}\label{s:proof}

In this section, we prove our two main results. It will be helpful to have some notation. As is usual, we write $[n]$ for the set $\{1, 2, \dots, n\}$. In the geodesic-biased random walk on a graph $G$, when the target vertex $b$ and set $\XX$ of excited vertices are clear from the context, we abbreviate the expected hitting time $\tau_x(y, \XX)$ of $y$ from $x$ by $T(x, y)$.

We shall make use of a well-known Chernoff-type bound.
\begin{proposition}\label{p:chernoff}
	Let $X=X_1 + X_2 + \dots + X_n$, where $X_1, X_2, \dots, X_n$ are independent Bernoulli random variables. Writing $\mu=\E[X]$, we have
	\[\Prob(X\geq (1+\delta)\mu)\leq\exp\left(\frac{-\delta^2\mu}{2+\delta}\right)\] for all $\delta>0$. \qed
\end{proposition}

We also require the following well-known gambler's ruin estimate.
\begin{proposition}\label{p:int}
	The probability that the simple random walk on the interval $\{0,1,\dots,n\}$ started at $1$ visits $n$ before it visits $0$ is $1/n$.\qed
\end{proposition}

We are now ready to give the proof of Theorem~\ref{thm:unbounded}.

\begin{proof}[Proof of Theorem~\ref{thm:unbounded}]
	We build an infinite family of graphs as follows. We fix $k\in\N$, set $m= \lfloor \sqrt{k} \rfloor$, and consider a graph $G$ as follows: we start with a path of length $m+1$ between $a$ and $b$, say $a,v_1,v_2\dots,v_m,b$, and then connect each $v_i$ to $a$ by $k$ disjoint paths of length $i+1$ as shown in Figure~\ref{fig:unbounded}. Formally, we take
	\[ V(G) =\{a,b\}\cup \{ v_1, v_2, \dots, v_m\} \cup\bigcup_{j=1}^m\bigcup_{i=1}^{j}R_{i,j},\] where
	$R_{i,j}=\{r_{i,j,l} : l\in [k]\}$, and specify $E(G)$ as follows:
	\begin{itemize}
		\item $\forall i\in [m-1]:\{v_i,v_{i+1}\}\in E(G)$,
		\item $\forall j\in[m], \forall i\in[j-1], \forall l\in [k]: \{r_{i,j,l}, r_{i,j+1,l}\}\in E(G)\land \{r_{i,1,l},a\}\in E(G)\land \{r_{i,k,l}, v_{i}\}\in E(G)$,
		\item $\{a,v_1\}\in E(G)$ and $\{v_m, b\}\in E(G)$.
	\end{itemize}

	We consider the geodesic-biased random walk on this graph with target $b$ and $\XX = \{a\}$. The unique shortest path to $b$ from $a$ is the path $a,v_1,v_2\dots,v_m,b$, so the random walker always moves to $v_1$ from $a$.

	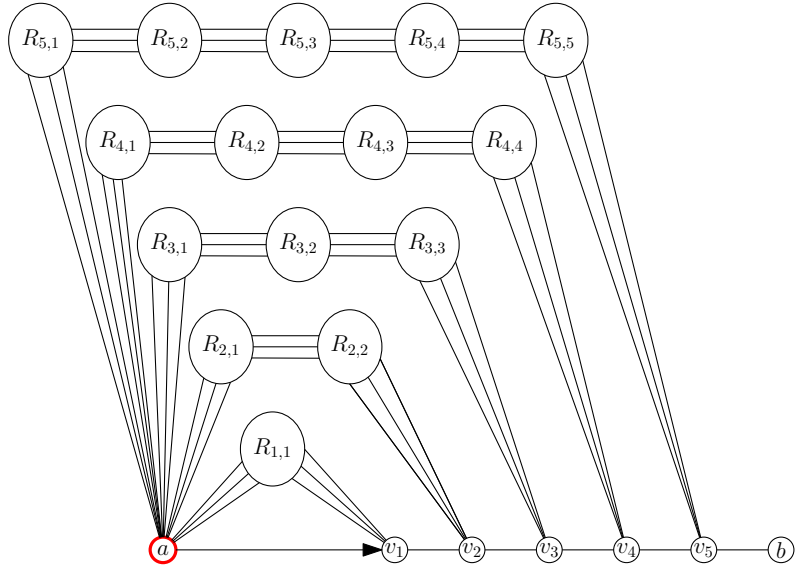
\begin{figure}
		\begin{center}
			\tikzstyle{ipe stylesheet} = [
  ipe import,
  even odd rule,
  line join=round,
  line cap=butt,
  ipe pen normal/.style={line width=0.4},
  ipe pen heavier/.style={line width=0.8},
  ipe pen fat/.style={line width=1.2},
  ipe pen ultrafat/.style={line width=2},
  ipe pen normal,
  ipe mark normal/.style={ipe mark scale=3},
  ipe mark large/.style={ipe mark scale=5},
  ipe mark small/.style={ipe mark scale=2},
  ipe mark tiny/.style={ipe mark scale=1.1},
  ipe mark normal,
  /pgf/arrow keys/.cd,
  ipe arrow normal/.style={scale=7},
  ipe arrow large/.style={scale=10},
  ipe arrow small/.style={scale=5},
  ipe arrow tiny/.style={scale=3},
  ipe arrow normal,
  /tikz/.cd,
  ipe arrows, 
  <->/.tip = ipe normal,
  ipe dash normal/.style={dash pattern=},
  ipe dash dashed/.style={dash pattern=on 4bp off 4bp},
  ipe dash dotted/.style={dash pattern=on 1bp off 3bp},
  ipe dash dash dotted/.style={dash pattern=on 4bp off 2bp on 1bp off 2bp},
  ipe dash dash dot dotted/.style={dash pattern=on 4bp off 2bp on 1bp off 2bp on 1bp off 2bp},
  ipe dash normal,
  ipe node/.append style={font=\normalsize},
  ipe stretch normal/.style={ipe node stretch=1},
  ipe stretch normal,
  ipe opacity 10/.style={opacity=0.1},
  ipe opacity 30/.style={opacity=0.3},
  ipe opacity 50/.style={opacity=0.5},
  ipe opacity 75/.style={opacity=0.75},
  ipe opacity opaque/.style={opacity=1},
  ipe opacity opaque,
]
\definecolor{red}{rgb}{1,0,0}
\definecolor{green}{rgb}{0,1,0}
\definecolor{blue}{rgb}{0,0,1}
\definecolor{yellow}{rgb}{1,1,0}
\definecolor{orange}{rgb}{1,0.647,0}
\definecolor{gold}{rgb}{1,0.843,0}
\definecolor{purple}{rgb}{0.627,0.125,0.941}
\definecolor{gray}{rgb}{0.745,0.745,0.745}
\definecolor{brown}{rgb}{0.647,0.165,0.165}
\definecolor{navy}{rgb}{0,0,0.502}
\definecolor{pink}{rgb}{1,0.753,0.796}
\definecolor{seagreen}{rgb}{0.18,0.545,0.341}
\definecolor{turquoise}{rgb}{0.251,0.878,0.816}
\definecolor{violet}{rgb}{0.933,0.51,0.933}
\definecolor{darkblue}{rgb}{0,0,0.545}
\definecolor{darkcyan}{rgb}{0,0.545,0.545}
\definecolor{darkgray}{rgb}{0.663,0.663,0.663}
\definecolor{darkgreen}{rgb}{0,0.392,0}
\definecolor{darkmagenta}{rgb}{0.545,0,0.545}
\definecolor{darkorange}{rgb}{1,0.549,0}
\definecolor{darkred}{rgb}{0.545,0,0}
\definecolor{lightblue}{rgb}{0.678,0.847,0.902}
\definecolor{lightcyan}{rgb}{0.878,1,1}
\definecolor{lightgray}{rgb}{0.827,0.827,0.827}
\definecolor{lightgreen}{rgb}{0.565,0.933,0.565}
\definecolor{lightyellow}{rgb}{1,1,0.878}
\definecolor{black}{rgb}{0,0,0}
\definecolor{white}{rgb}{1,1,1}
\begin{tikzpicture}[ipe stylesheet, scale=0.6, every node/.style={scale=0.6}]
  \draw
    (84.04, 807.215) ellipse[x radius=19.96, y radius=23.2095];
  \draw
    (164.04, 807.215) ellipse[x radius=19.96, y radius=23.2095];
  \node[ipe node, anchor=center, font=\large]
     at (84.04, 807.209) {$R_{5,1}
$};
  \draw
    (104, 807.211)
     -- (144.08, 807.214);
  \draw
    (244.04, 807.215) ellipse[x radius=19.96, y radius=23.2095];
  \node[ipe node, anchor=center, font=\large]
     at (164.04, 807.209) {$R_{5,2}
$};
  \node[ipe node, anchor=center, font=\large]
     at (244.04, 807.215) {$R_{5,3}$};
  \draw
    (324.04, 807.215) ellipse[x radius=19.96, y radius=23.2095];
  \node[ipe node, anchor=center, font=\large]
     at (324.04, 807.215) {$R_{5,4}$};
  \draw
    (404.04, 807.215) ellipse[x radius=19.96, y radius=23.2095];
  \node[ipe node, anchor=center, font=\large]
     at (404.04, 807.215) {$R_{5,5}$};
  \draw
    (164.04, 679.215) ellipse[x radius=19.96, y radius=23.2095];
  \draw
    (244.04, 679.215) ellipse[x radius=19.96, y radius=23.2095];
  \node[ipe node, anchor=center, font=\large]
     at (164.04, 679.21) {$R_{3,1}
$};
  \draw
    (324.04, 679.215) ellipse[x radius=19.96, y radius=23.2095];
  \node[ipe node, anchor=center, font=\large]
     at (244.04, 679.21) {$R_{3,2}
$};
  \node[ipe node, anchor=center, font=\large]
     at (324.04, 679.215) {$R_{3,3}$};
  \draw
    (132.04, 743.215) ellipse[x radius=19.96, y radius=23.2095];
  \draw
    (212.04, 743.215) ellipse[x radius=19.96, y radius=23.2095];
  \node[ipe node, anchor=center, font=\large]
     at (132.04, 743.21) {$R_{4,1}
$};
  \node[ipe node, anchor=center, font=\large]
     at (212.04, 743.215) {$R_{4,2}$};
  \draw
    (292.04, 743.215) ellipse[x radius=19.96, y radius=23.2095];
  \node[ipe node, anchor=center, font=\large]
     at (292.04, 743.215) {$R_{4,3}$};
  \draw
    (372.04, 743.215) ellipse[x radius=19.96, y radius=23.2095];
  \node[ipe node, anchor=center, font=\large]
     at (372.04, 743.215) {$R_{4,4}$};
  \draw
    (196.04, 615.215) ellipse[x radius=19.96, y radius=23.2095];
  \draw
    (276.04, 615.215) ellipse[x radius=19.96, y radius=23.2095];
  \node[ipe node, anchor=center, font=\large]
     at (196.04, 615.21) {$R_{2,1}
$};
  \node[ipe node, anchor=center, font=\large]
     at (276.04, 615.215) {$R_{2,2}$};
  \draw
    (228.04, 551.215) ellipse[x radius=19.96, y radius=23.2095];
  \node[ipe node, anchor=center, font=\large]
     at (228.04, 551.215) {$R_{1,1}$};
  \draw[red, ipe pen fat]
    (160, 488) circle[radius=8];
  \draw
    (352, 488) circle[radius=8];
  \node[ipe node, anchor=center, font=\large]
     at (160, 488) {$a
$};
  \node[ipe node, anchor=center, font=\large]
     at (352, 488) {$v_2$};
  \draw
    (400, 488) circle[radius=8];
  \node[ipe node, anchor=center, font=\large]
     at (400, 488) {$v_3$};
  \draw
    (496, 488) circle[radius=8];
  \node[ipe node, anchor=center, font=\large]
     at (496, 488) {$v_5$};
  \draw
    (360, 488)
     -- (392, 488);
  \draw
    (408, 488)
     -- (440, 488);
  \draw
    (456, 488)
     -- (488, 488);
  \draw
    (544, 488) circle[radius=8];
  \node[ipe node, anchor=center, font=\large]
     at (544, 488) {$b$};
  \draw
    (536, 488)
     -- (504, 488);
  \draw
    (448, 488) circle[radius=8];
  \node[ipe node, anchor=center, font=\large]
     at (448, 488) {$v_4$};
  \draw
    (304, 488) circle[radius=8];
  \node[ipe node, anchor=center, font=\large]
     at (304, 488) {$v_1$};
  \draw
    (312, 488)
     -- (344, 488);
  \draw[->]
    (168, 488)
     -- (296, 488);
  \draw
    (287.424, 596.15)
     -- (347.899, 494.869);
  \draw
    (332.41, 658.145)
     -- (397.047, 495.435);
  \draw
    (378.568, 721.282)
     -- (445.718, 495.668);
  \draw
    (493.785, 495.687)
     -- (410.38, 785.207);
  \draw
    (298.685, 493.979)
     -- (247.999, 551.001);
  \draw
    (298.054, 493.352)
     -- (245.723, 540.449);
  \draw
    (240.106, 532.726)
     -- (297.446, 492.588);
  \draw
    (219.817, 530.067)
     -- (166.554, 492.588);
  \draw
    (165.315, 493.979)
     -- (209.245, 543.401);
  \draw
    (165.946, 493.352)
     -- (213.115, 535.804);
  \draw
    (185.325, 595.633)
     -- (161.832, 495.787);
  \draw
    (162.409, 495.629)
     -- (192.933, 592.288);
  \draw
    (201.922, 593.036)
     -- (162.966, 495.43);
  \draw
    (160.638, 495.975)
     -- (173.671, 658.886);
  \draw
    (163.55, 656.012)
     -- (160.169, 495.998);
  \draw
    (159.681, 495.994)
     -- (153.134, 659.776);
  \draw
    (129.031, 720.271)
     -- (158.943, 495.93);
  \draw
    (159.129, 495.952)
     -- (134.562, 720.192)
     -- (134.562, 720.192);
  \draw
    (122.21, 723.015)
     -- (158.73, 495.899);
  \draw
    (157.832, 495.701)
     -- (76.1211, 785.91);
  \draw
    (158.148, 495.783)
     -- (89.3617, 784.846);
  \draw
    (97.9798, 790.603)
     -- (158.394, 495.837);
  \draw
    (103.076, 800.234)
     -- (145.052, 800.061);
  \draw
    (144.993, 814.155)
     -- (103.087, 814.155);
  \draw
    (184, 807.211)
     -- (224.08, 807.214);
  \draw
    (183.076, 800.234)
     -- (225.052, 800.061);
  \draw
    (224.993, 814.155)
     -- (183.087, 814.155);
  \draw
    (264, 807.211)
     -- (304.08, 807.214);
  \draw
    (263.076, 800.234)
     -- (305.052, 800.061);
  \draw
    (304.993, 814.155)
     -- (263.087, 814.155);
  \draw
    (344, 807.211)
     -- (384.08, 807.214);
  \draw
    (343.076, 800.234)
     -- (385.052, 800.061);
  \draw
    (384.993, 814.155)
     -- (343.087, 814.155);
  \draw
    (312, 743.211)
     -- (352.08, 743.214);
  \draw
    (311.076, 736.234)
     -- (353.052, 736.061);
  \draw
    (352.993, 750.155)
     -- (311.087, 750.155);
  \draw
    (232, 743.211)
     -- (272.08, 743.214);
  \draw
    (231.076, 736.234)
     -- (273.052, 736.061);
  \draw
    (272.993, 750.155)
     -- (231.087, 750.155);
  \draw
    (152, 743.211)
     -- (192.08, 743.214);
  \draw
    (151.076, 736.234)
     -- (193.052, 736.061);
  \draw
    (192.993, 750.155)
     -- (151.087, 750.155);
  \draw
    (184, 679.211)
     -- (224.08, 679.214);
  \draw
    (183.076, 672.234)
     -- (225.052, 672.061);
  \draw
    (224.993, 686.155)
     -- (183.087, 686.155);
  \draw
    (264, 679.211)
     -- (304.08, 679.214);
  \draw
    (263.076, 672.234)
     -- (305.052, 672.061);
  \draw
    (304.993, 686.155)
     -- (263.087, 686.155);
  \draw
    (216, 615.211)
     -- (256.08, 615.214);
  \draw
    (215.076, 608.234)
     -- (257.052, 608.061);
  \draw
    (256.993, 622.155)
     -- (215.087, 622.155);
  \draw
    (348.556, 495.221)
     -- (294.915, 607.666);
  \draw
    (275.969, 592.006)
     -- (347.279, 494.458);
  \draw
    (494.094, 495.77)
     -- (420.797, 794.604);
  \draw
    (396.995, 785.499)
     -- (493.474, 495.591);
  \draw
    (446.103, 495.772)
     -- (388.781, 730.576);
  \draw
    (365.366, 721.341)
     -- (365.366, 721.341)
     -- (445.329, 495.541);
  \draw
    (319.557, 656.598)
     -- (396.555, 495.22);
  \draw
    (397.541, 495.613)
     -- (341.721, 668.446);
  \draw
    (294.915, 607.666)
     -- (348.556, 495.221);
  \draw
    (347.279, 494.458)
     -- (275.969, 592.006);
\end{tikzpicture}
		\end{center}
		\caption{The construction with $m=5$.}
		\label{fig:unbounded}
	\end{figure}

	\begin{lemma}
		\label{lem:expectation}
		For $1\leq j\leq m+1$, we have $T(a,v_j)\geq\tfrac{k^{j-1}}{4^{j-1}\cdot (j-1)!}$.
	\end{lemma}
	\begin{proof}
		We will prove this lemma by induction. For $j=1$, we have $T(a,v_1)=1$ and the bound clearly holds.
		Now, assume the lemma holds for $j$ and note that $T(a,v_{j+1})=T(a,v_j)+T(v_j,v_{j+1})$, as we can only reach $v_{j+1}$ from $v_j$. We may then bound $T(v_j,v_{j+1})$ by
		\begin{align*}
			T(v_j,v_{j+1}) & =1+\frac{1}{k+2}T(v_{j-1},v_{j+1})+\frac{1}{k+2}T(v_{j+1},v_{j+1})+\frac{k}{k+2}T(R_{j,j},v_{j+1}) \\
			               & \geq\frac{k}{k+2}T(R_{j,j},v_{j+1})
		\end{align*}
		From Proposition~\ref{p:int}, it follows that the probability of walking from $R_{j,j}$ to $q_{j+1}$ before $a$ is ${j}/{(j+1)}$, and the complementary event has the probability ${1}/{(j+1)}$. We then see that \[T(R_{j,j},v_{j+1})\geq \frac{1}{j+1}T(a,v_{j+1})+\frac{j}{j+1}T(v_j,v_{j+1}).\]
		Using this bound, we obtain
		\begin{align*}
			T(v_j,v_{j+1})                          & \geq \frac{k}{k+2}\frac{1}{j+1}T(a,v_{j+1})+\frac{k}{k+2}\frac{j}{j+1} T(v_j,v_{j+1}), \text{ so } \\
			\frac{k+2j+2}{(k+2)(j+1)}T(v_j,v_{j+1}) & \geq \frac{k}{(k+2)(j+1)}T(a,v_{j+1}), \text{ whence }                                             \\
			T(v_j,v_{j+1})                          & \geq \frac{k}{k+2j+2}T(a,v_{j+1})
		\end{align*}
		Combining the above bound with the bound on $T(a,v_{j+1})$, we get
		\begin{align*}
			T(a,v_{j+1})                    & \geq T(a,v_{j})+\frac{k}{k+2j+2}T(a,v_{j+1}), \text{ so }    \\
			\frac{2j+2}{k+2j+2}T(a,v_{j+1}) & \geq T(a,v_j), \text{ whence }                               \\
			T(a,v_{j+1})                    & \geq \frac{k+2j+2}{2j+2}T(a,v_j) \geq \frac{k}{4j}T(a,v_{j})
		\end{align*}
		By the induction hypothesis, we now conclude that
		\[T(a,v_{j+1}) \geq \frac{k}{4j} \frac{k^{j-1}}{4^{j-1}\cdot (j-1)!}=\frac{k^{j}}{4^{j}\cdot j!};\]
		the result follows.
	\end{proof}

	From Lemma \ref{lem:expectation}, we conclude that $T(a,b) \geq {k^m}/{4^mm!}$; since $m = \lfloor \sqrt{k} \rfloor$, standard bounds for the factorial show that
	\[T(a,b) \geq \frac{1}{4} \left(\frac{\sqrt k}{4}\right)^{\sqrt k - 1}\]  and since $n = |V(G)|=\Theta(m^2 k)=\Theta(k^2)$, we deduce that
	\[T(a,b) =\Omega\left(\exp\left(\frac{\sqrt[4]{n} \log n}{100}\right)\right),\]
	proving the result.
\end{proof}

Next, we present the (slightly more involved) proof of Theorem~\ref{thm:bounded}.

\begin{proof}[Proof of Theorem~\ref{thm:bounded}]
	To prove the result, we build an infinite family of graphs as follows. We fix $m\in\N$, and consider a graph $G$ constructed as follows: as before, we start with a path of length $m+1$ between $a$ and $b$, say $a,v_1,v_2\dots,v_m,b$, and then attach a path of length $2m + 2$ to each $v_i$, and finally chain the ends of these paths to $a$ by another path as shown in Figure~\ref{fig:bounded}. Formally, we set
	\[V(G)=\{a,b\}\cup \{v_1, v_2,\dots,v_m\}\cup \{s_1, s_2, \dots, s_m\}\cup\bigcup_{j=1}^{2m+1}\bigcup_{i=1}^{m}\{r_{i,j}\}\]
	and specify $E(G)$ as follows:
	\begin{itemize}
		\item $\forall i\in [m-1]:\{v_i,v_{i+1}\}\in E(G) \land \{s_i,s_{i+1}\}\in E(G)$,
		\item $\forall j\in[2m], \forall i\in[m]: \{r_{i,j}, r_{i,j+1}\}\in E (G) \land \{r_{i,1},s_i\}\in E(G) \land \{r_{i,2m+1}, v_{i}\}\in E(G)$,
		\item $\{a,v_1\}\in E(G)$, $\{v_m, b\}\in E(G)$, and $\{a,s_1\}\in E(G)$.
	\end{itemize}
	We consider the geodesic-biased random walk on this graph with target $b$ and $\XX =  \{a, s_1, s_2, \dots, s_m\}$. Notice that our choice of path lengths ensures that the random walker moves deterministically from $s_i$ to $s_{i-1}$ (or to $a$ in the case of $s_1$), and from $a$ to $v_1$.

	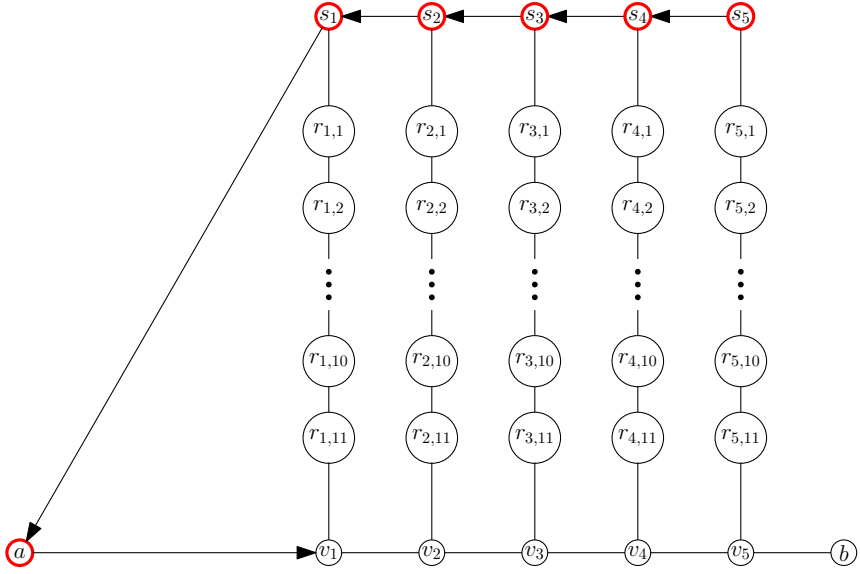
\begin{figure}
		\begin{center}
			\tikzstyle{ipe stylesheet} = [
  ipe import,
  even odd rule,
  line join=round,
  line cap=butt,
  ipe pen normal/.style={line width=0.4},
  ipe pen heavier/.style={line width=0.8},
  ipe pen fat/.style={line width=1.2},
  ipe pen ultrafat/.style={line width=2},
  ipe pen normal,
  ipe mark normal/.style={ipe mark scale=3},
  ipe mark large/.style={ipe mark scale=5},
  ipe mark small/.style={ipe mark scale=2},
  ipe mark tiny/.style={ipe mark scale=1.1},
  ipe mark normal,
  /pgf/arrow keys/.cd,
  ipe arrow normal/.style={scale=7},
  ipe arrow large/.style={scale=10},
  ipe arrow small/.style={scale=5},
  ipe arrow tiny/.style={scale=3},
  ipe arrow normal,
  /tikz/.cd,
  ipe arrows, 
  <->/.tip = ipe normal,
  ipe dash normal/.style={dash pattern=},
  ipe dash dashed/.style={dash pattern=on 4bp off 4bp},
  ipe dash dotted/.style={dash pattern=on 1bp off 3bp},
  ipe dash dash dotted/.style={dash pattern=on 4bp off 2bp on 1bp off 2bp},
  ipe dash dash dot dotted/.style={dash pattern=on 4bp off 2bp on 1bp off 2bp on 1bp off 2bp},
  ipe dash normal,
  ipe node/.append style={font=\normalsize},
  ipe stretch normal/.style={ipe node stretch=1},
  ipe stretch normal,
  ipe opacity 10/.style={opacity=0.1},
  ipe opacity 30/.style={opacity=0.3},
  ipe opacity 50/.style={opacity=0.5},
  ipe opacity 75/.style={opacity=0.75},
  ipe opacity opaque/.style={opacity=1},
  ipe opacity opaque,
]
\definecolor{red}{rgb}{1,0,0}
\definecolor{green}{rgb}{0,1,0}
\definecolor{blue}{rgb}{0,0,1}
\definecolor{yellow}{rgb}{1,1,0}
\definecolor{orange}{rgb}{1,0.647,0}
\definecolor{gold}{rgb}{1,0.843,0}
\definecolor{purple}{rgb}{0.627,0.125,0.941}
\definecolor{gray}{rgb}{0.745,0.745,0.745}
\definecolor{brown}{rgb}{0.647,0.165,0.165}
\definecolor{navy}{rgb}{0,0,0.502}
\definecolor{pink}{rgb}{1,0.753,0.796}
\definecolor{seagreen}{rgb}{0.18,0.545,0.341}
\definecolor{turquoise}{rgb}{0.251,0.878,0.816}
\definecolor{violet}{rgb}{0.933,0.51,0.933}
\definecolor{darkblue}{rgb}{0,0,0.545}
\definecolor{darkcyan}{rgb}{0,0.545,0.545}
\definecolor{darkgray}{rgb}{0.663,0.663,0.663}
\definecolor{darkgreen}{rgb}{0,0.392,0}
\definecolor{darkmagenta}{rgb}{0.545,0,0.545}
\definecolor{darkorange}{rgb}{1,0.549,0}
\definecolor{darkred}{rgb}{0.545,0,0}
\definecolor{lightblue}{rgb}{0.678,0.847,0.902}
\definecolor{lightcyan}{rgb}{0.878,1,1}
\definecolor{lightgray}{rgb}{0.827,0.827,0.827}
\definecolor{lightgreen}{rgb}{0.565,0.933,0.565}
\definecolor{lightyellow}{rgb}{1,1,0.878}
\definecolor{black}{rgb}{0,0,0}
\definecolor{white}{rgb}{1,1,1}
\begin{tikzpicture}[ipe stylesheet, scale=0.6, every node/.style={scale=0.6}]
  \draw
    (224, 456) circle[radius=8];
  \node[ipe node, anchor=center, font=\large]
     at (224, 456) {$v_1$};
  \draw
    (480, 456) circle[radius=8];
  \node[ipe node, anchor=center, font=\large]
     at (480, 456) {$v_5$};
  \draw
    (288, 456) circle[radius=8];
  \node[ipe node, anchor=center, font=\large]
     at (288, 456) {$v_2$};
  \draw
    (352, 456) circle[radius=8];
  \node[ipe node, anchor=center, font=\large]
     at (352, 456) {$v_3$};
  \draw
    (416, 456) circle[radius=8];
  \node[ipe node, anchor=center, font=\large]
     at (416, 456) {$v_4$};
  \draw
    (544, 456) circle[radius=8];
  \node[ipe node, anchor=center, font=\large]
     at (544, 456) {$b$};
  \node[ipe node, anchor=center, font=\large]
     at (224, 528) {$r_{1,11}$};
  \draw
    (232, 456)
     -- (280, 456);
  \draw
    (296, 456)
     -- (344, 456);
  \draw
    (360, 456)
     -- (408, 456);
  \draw
    (424, 456)
     -- (472, 456);
  \draw
    (488, 456)
     -- (536, 456);
  \node[ipe node, anchor=center, font=\large]
     at (224, 576) {$r_{1,10}$};
  \node[ipe node, anchor=center, font=\large]
     at (224, 672) {$r_{1,2}$};
  \node[ipe node, anchor=center, font=\large]
     at (224, 720) {$r_{1,1}$};
  \draw[red, ipe pen fat]
    (224, 792) circle[radius=8];
  \node[ipe node, anchor=center, font=\large]
     at (224, 792) {$s_1$};
  \draw
    (480, 792) circle[radius=8];
  \node[ipe node, anchor=center, font=\large]
     at (480, 792) {$s_5$};
  \draw
    (288, 792) circle[radius=8];
  \node[ipe node, anchor=center, font=\large]
     at (288, 792) {$s_2$};
  \draw
    (352, 792) circle[radius=8];
  \node[ipe node, anchor=center, font=\large]
     at (352, 792) {$s_3$};
  \draw
    (416, 792) circle[radius=8];
  \node[ipe node, anchor=center, font=\large]
     at (416, 792) {$s_4$};
  \draw[<-]
    (232, 792)
     -- (280, 792);
  \draw[<-]
    (296, 792)
     -- (344, 792);
  \draw[<-]
    (360, 792)
     -- (408, 792);
  \draw[<-]
    (424, 792)
     -- (472, 792);
  \draw[red, ipe pen fat]
    (32, 456) circle[radius=8];
  \node[ipe node, anchor=center, font=\large]
     at (32, 456) {$a
$};
  \draw[->]
    (40, 456)
     -- (216, 456);
  \draw[->]
    (220.031, 785.054)
     -- (35.9691, 462.946);
  \draw
    (224, 512.528)
     -- (224, 464);
  \draw[red, ipe pen fat]
    (288, 792) circle[radius=8];
  \draw[red, ipe pen fat]
    (352, 792) circle[radius=8];
  \draw[red, ipe pen fat]
    (416, 792) circle[radius=8];
  \draw[red, ipe pen fat]
    (480, 792) circle[radius=8];
  \pic
     at (224, 624) {ipe disk};
  \pic
     at (224, 632) {ipe disk};
  \pic
     at (224, 616) {ipe disk};
  \draw
    (224, 528) circle[radius=16];
  \draw
    (224, 576) circle[radius=16];
  \draw
    (224, 560)
     -- (224, 544)
     -- (224, 544);
  \draw
    (224, 592)
     -- (224, 608);
  \draw
    (224, 672) circle[radius=16];
  \draw
    (224, 720) circle[radius=16];
  \draw
    (224, 656)
     -- (224, 640);
  \draw
    (224, 688)
     -- (224, 704);
  \draw
    (224, 784)
     -- (224, 736);
  \node[ipe node, anchor=center, font=\large]
     at (288, 528) {$r_{2,11}$};
  \node[ipe node, anchor=center, font=\large]
     at (288, 576) {$r_{2,10}$};
  \node[ipe node, anchor=center, font=\large]
     at (288, 672) {$r_{2,2}$};
  \node[ipe node, anchor=center, font=\large]
     at (288, 720) {$r_{2,1}$};
  \draw
    (288, 512.528)
     -- (288, 464);
  \pic
     at (288, 624) {ipe disk};
  \pic
     at (288, 632) {ipe disk};
  \pic
     at (288, 616) {ipe disk};
  \draw
    (288, 528) circle[radius=16];
  \draw
    (288, 576) circle[radius=16];
  \draw
    (288, 560)
     -- (288, 544)
     -- (288, 544);
  \draw
    (288, 592)
     -- (288, 608);
  \draw
    (288, 672) circle[radius=16];
  \draw
    (288, 720) circle[radius=16];
  \draw
    (288, 656)
     -- (288, 640);
  \draw
    (288, 688)
     -- (288, 704);
  \draw
    (288, 784)
     -- (288, 736);
  \node[ipe node, anchor=center, font=\large]
     at (352, 528) {$r_{3,11}$};
  \node[ipe node, anchor=center, font=\large]
     at (352, 576) {$r_{3,10}$};
  \node[ipe node, anchor=center, font=\large]
     at (352, 672) {$r_{3,2}$};
  \node[ipe node, anchor=center, font=\large]
     at (352, 720) {$r_{3,1}$};
  \draw
    (352, 512.528)
     -- (352, 464);
  \pic
     at (352, 624) {ipe disk};
  \pic
     at (352, 632) {ipe disk};
  \pic
     at (352, 616) {ipe disk};
  \draw
    (352, 528) circle[radius=16];
  \draw
    (352, 576) circle[radius=16];
  \draw
    (352, 560)
     -- (352, 544)
     -- (352, 544);
  \draw
    (352, 592)
     -- (352, 608);
  \draw
    (352, 672) circle[radius=16];
  \draw
    (352, 720) circle[radius=16];
  \draw
    (352, 656)
     -- (352, 640);
  \draw
    (352, 688)
     -- (352, 704);
  \draw
    (352, 784)
     -- (352, 736);
  \node[ipe node, anchor=center, font=\large]
     at (416, 528) {$r_{4,11}$};
  \node[ipe node, anchor=center, font=\large]
     at (416, 576) {$r_{4,10}$};
  \node[ipe node, anchor=center, font=\large]
     at (416, 672) {$r_{4,2}$};
  \node[ipe node, anchor=center, font=\large]
     at (416, 720) {$r_{4,1}$};
  \draw
    (416, 512.528)
     -- (416, 464);
  \pic
     at (416, 624) {ipe disk};
  \pic
     at (416, 632) {ipe disk};
  \pic
     at (416, 616) {ipe disk};
  \draw
    (416, 528) circle[radius=16];
  \draw
    (416, 576) circle[radius=16];
  \draw
    (416, 560)
     -- (416, 544)
     -- (416, 544);
  \draw
    (416, 592)
     -- (416, 608);
  \draw
    (416, 672) circle[radius=16];
  \draw
    (416, 720) circle[radius=16];
  \draw
    (416, 656)
     -- (416, 640);
  \draw
    (416, 688)
     -- (416, 704);
  \draw
    (416, 784)
     -- (416, 736);
  \node[ipe node, anchor=center, font=\large]
     at (480, 528) {$r_{5,11}$};
  \node[ipe node, anchor=center, font=\large]
     at (480, 576) {$r_{5,10}$};
  \node[ipe node, anchor=center, font=\large]
     at (480, 672) {$r_{5,2}$};
  \node[ipe node, anchor=center, font=\large]
     at (480, 720) {$r_{5,1}$};
  \draw
    (480, 512.528)
     -- (480, 464);
  \pic
     at (480, 624) {ipe disk};
  \pic
     at (480, 632) {ipe disk};
  \pic
     at (480, 616) {ipe disk};
  \draw
    (480, 528) circle[radius=16];
  \draw
    (480, 576) circle[radius=16];
  \draw
    (480, 560)
     -- (480, 544)
     -- (480, 544);
  \draw
    (480, 592)
     -- (480, 608);
  \draw
    (480, 672) circle[radius=16];
  \draw
    (480, 720) circle[radius=16];
  \draw
    (480, 656)
     -- (480, 640);
  \draw
    (480, 688)
     -- (480, 704);
  \draw
    (480, 784)
     -- (480, 736);
\end{tikzpicture}
			\caption{The bounded-degree construction with $m = 5$.}
			\label{fig:bounded}
		\end{center}
	\end{figure}

	\begin{lemma} We have $T(v_{1},b)\geq {\exp(\sqrt{m}/10)}/({m^{3/2}+1})$.
	\end{lemma}

	\begin{proof}
		We proceed via a renewal argument. Observe that $T(v_{1},b)\geq 1+q\cdot T(v_{1},b)$, where $q$ is the probability of the event that the random walker visits $a$ before $b$ after leaving $v_1$. It will be more convenient to work with the complementary event, namely, that the random walker visits $b$ before $a$ after leaving $v_1$; we write $p=1-q$ for the probability of this event. From the previous inequality, we then have $T(v_{1},b)\geq {1}/{(1-q)}=1/p$.

		Now, we shall estimate $p$, the probability that the geodesic-biased walk starting at $v_1$ hits $b$ before $a$. To do so, we consider the Markov chain $(x_t)_{t\ge 0}$ induced by the geodesic-biased walk on the states $a, v_1, \dots, v_m, b$ with $a$ and $b$ being absorbing; of course, $p$ is exactly the probability that this induced chain started at $v_1$ reaches the absorbing state $b$ before it hits the absorbing state $a$.

		For each non-absorbing state $v_i$, there are three possibilities for the next state of the induced chain hit by the random-walker: $v_{i-1}$, $v_{i+1}$ or $a$. The probabilities of these transitions are as follows: we write $\varepsilon$ for the probability of returning to $a$ via $s_i$, and note that the other two transitions have the same probability, i.e.,
		\[\Prob[x_{t+1}=v_{i+1}\cond x_t=v_i]=\Prob[x_{t+1}=v_{i-1} \cond x_t=v_i]=\frac{1-\varepsilon}{2}.\]
		We may calculate $\varepsilon$, the probability of \emph{retracing}, i.e., returning to $a$ via $s_i$, as follows. The probability of reaching $s_i$ before $v_i$ starting from $r_{i,2k+1}$ is, by Proposition~\ref{p:int}, exactly ${1}/{(2m+2)}$. It then follows that $\varepsilon=\tfrac{1}{3}(\frac{2m+1}{2m+2}\varepsilon + \tfrac{1}{2m+2})$, from which we get $\varepsilon={1}/{(4m+5)}$.

		We shall estimate $p=p_s+p_l$ by separately estimating $p_s$, the probability of the chain hitting $b$ before $a$ starting from $v_1$ in at most $m^{3/2}$ steps, and $p_l$, the probability of the chain hitting $b$ before $a$ starting from $v_1$ and taking more than $m^{3/2}$ steps to do so.

		First, we dispose of `long' excursions. We claim that $p_l\leq (1-\varepsilon)^{m^{3/2}}$; indeed, if the chain does not hit either of $a$ or $b$ in the first $m^{3/2}$ steps, then the chain does not, in particular, retrace on any of the first $m^{3/2}$ steps. Thus
		\[p_l\leq (1-\varepsilon)^{m^{3/2}}\leq \left(1-\frac{1}{4m+5}\right)^{m^{3/2}}\leq \exp\left(\frac{-\sqrt{m}}{10}\right).\]

		Next, we focus on the `short' excursions. Note that we may write $p_s=\sum_{t=0}^{m^{3/2}}p(t)$, where \[ p(t)=\Prob[\{x_t=b\}\land \{\forall\,1\leq i<t: (x_i\neq a\land x_i\neq b)\}].\] We may then bound $p(t)$ by conditioning on the chain never retracing to get
		\[ p(t)\leq \Prob [\{x_t=b\}\land \{\forall\,1\leq i<t: (x_i\neq a\land x_i\neq b)\} \cond \text{No Retrace}].\]
		This upper bound may be interpreted in terms of the simple random walk on the integers; indeed, conditional on never retracing, the chain is isomorphic to the simple random walk on the integer line. Concretely, consider the simple random walk $\{y_t\}_{t\ge=0}$ on the integers and note that
		\begin{align*}
			\Prob  [\{x_t=b\} & \land \{\forall\,1\leq i<t: (x_i\neq a\land x_i\neq b)\} \cond \text{No Retrace}]           \\
			                  & = \Prob[\{y_0=1\land y_t = m+1\}\land \{\forall 1\leq i<t: (y_i\neq 0\land y_i\neq m+1) \}] \\
			                  & \leq \Prob[\{y_0=1\land y_t=m+1\}]\leq \Prob[\{y_0=1\land y_t\geq m+1\}].
		\end{align*}
		The last probability above is easy to estimate since the simple random walk on the integers may be viewed as a sum of independent Bernoulli random variables, so by applying Proposition \ref{p:chernoff} (with $\delta={m}/{t}$) to such a representation of the random walk on the integers, we obtain
		\[\Prob[\{y_0=1\land y_t\geq m+1\}] \le \exp\left(\frac{-m^2}{4t+2m}\right) \le \exp\left(\frac{-\sqrt{m}}{10}\right),\]
		where the second inequality holds for all $t\leq m^{3/2}$. Consequently, we have
		\[p_s \le m^{3/2} \exp\left(\frac{-\sqrt{m}}{10}\right).\]

		Combining the above estimates for $p_s$ and $p_l$ and the fact that $T(v_1, b) \ge 1/(p_s + p_l)$ now yields the required bound.
	\end{proof}

	The theorem immediately follows from the above lemma. Indeed, $T(a,b) = 1 + T(v_1, b)$, and writing the above bound for $T(v_1, b)$ in terms of $n = |V(G)| = 2+m(2m+3)$ proves the result.
\end{proof}

\section{Conclusion}\label{s:conc}
Our results raise a few different natural questions; we discuss two such problems below.

There remains the question of determining the right order of uniform bound for the expected hitting time of a fixed target in the geodesic-biased walk: we have shown that on a connected $n$-vertex graph, this may be as large as $\exp(n^{1/4}\log n/100)$, while it is more or less trivial to show a uniform upper bound of $\exp(n\log n)$; it would be interesting to close this gap and pin down the truth.

Another problem that we have been unable to resolve concerns bounded-degree graphs. While we have exhibited exponential slowdown for the geodesic-biased walk on bounded-degree graphs, our constructions nonetheless require an unbounded number of excitations, which leads to the following: in the geodesic-biased walk on a bounded-degree graph with a bounded number of excitations, is there a uniform polynomial bound on the expected hitting time of the fixed target?

\section*{Acknowledgements}
The first, second and fourth authors were supported by H2020-MSCA-RISE project CoSP 823748, and the third author wishes to acknowledge support from NSF grant DMS-1800521. Much of this work was carried out when the first, second and fourth authors were participants in the DIMACS REU supported by NSF grant CCF-1852215; we are grateful for the hospitality of the DIMACS Center.
\bibliographystyle{amsplain}
\bibliography{geodesic_walk}
\end{document}